\documentclass[a4paper,10pt]{amsart}

\usepackage{enumerate, booktabs}
\usepackage{amsmath, amsfonts, amssymb, amsthm}
\usepackage[height=22.5cm, width=16cm]{geometry}
\usepackage[all]{xy}

\theoremstyle{plain}
\newtheorem{theorem}{Theorem}[section]
\newtheorem{proposition}[theorem]{Proposition}
\newtheorem{lemma}[theorem]{Lemma}
\newtheorem{corollary}[theorem]{Corollary}
\newtheorem*{thm}{Theorem}

\theoremstyle{remark}
\newtheorem{definition}[theorem]{Definition}
\newtheorem*{remark}{Remark}
\newtheorem*{example}{Example}

\newtheorem*{acknowledgement}{Acknowledgements}

\numberwithin{equation}{section}

\providecommand{\bysame}{\leavevmode ---\ }
\providecommand{\og}{``}
\providecommand{\fg}{''}
\providecommand{\smfandname}{\&}

\providecommand{\smfedname}{\'ed.}

\newcommand{\mbb}[1]{\mathbb{#1}}
\newcommand{\lie}[1]{{\mathfrak{#1}}}
\newcommand{\abs}[1]{\lvert #1\rvert}
\newcommand{\norm}[1]{\lVert #1\rVert}

\newcommand{\vf}[1]{\frac{\partial}{\partial #1}}

\DeclareMathOperator{\im}{Im}
\DeclareMathOperator{\id}{id}
\DeclareMathOperator{\ad}{ad}
\DeclareMathOperator{\Ad}{Ad}
\DeclareMathOperator{\Aut}{Aut}
\DeclareMathOperator{\Lie}{Lie}

\title{Quotients of bounded homogeneous domains by cyclic groups}
\author{Christian Miebach}\thanks{The author is supported by SFB/TR 12 of the
Deutsche Forschungsgemeinschaft.}
\address{Fakult\"at f\"ur Mathematik, Ruhr-Universit\"at Bochum,
Universit\"atsstra{\ss}e 150, D - 44780 Bochum, Telefon: +49 - 234 - 32 - 23329,
Fax: +49 - 234 - 32 - 14498}
\email{christian.miebach@ruhr-uni-bochum.de}


\begin{document}

\maketitle

\begin{abstract}
Let $D$ be a bounded homogeneous domain in $\mbb{C}^n$ and let $\varphi$ be an
automorphism of $D$ which generates a discrete subgroup $\Gamma$ of
$\Aut_\mathcal{O}(D)$. It is shown that the complex space $D/\Gamma$ is Stein.
\end{abstract}

\section{Introduction}

Let $D\subset\mbb{C}^n$ be a bounded domain of holomorphy and let $\varphi$ be
an automorphism of $D$ such that the cyclic group $\Gamma:=\langle\varphi
\rangle:=\bigl\{\varphi^k;\ k\in\mbb{Z}\bigr\}$ is a discrete subgroup of the
automorphism group $\Aut_\mathcal{O}(D)$. It follows that $\Gamma$ acts
properly on $D$ and hence that the quotient $X:=D/\Gamma$ is a complex space.
In this situation one would like to know conditions on $D$ or $\varphi$ which
guarantee that $X$ is a Stein space.

Since the group $\Gamma$ is cyclic, it is either finite or isomorphic to
$\mbb{Z}$. In the first case it is a classical result that Steinness of $D$
implies Steinness of $X$. Therefore we assume that $\Gamma$ is infinite cyclic.
In the case that $D$ is biholomorphically equivalent to the unit ball
$\mbb{B}_n$ it is proven in~\cite{Fab} and~\cite{FabIan} that $X=D/\langle
\varphi\rangle$ is Stein for hyperbolic and parabolic automorphisms $\varphi$.
We will generalize this result to arbitrary bounded homogeneous domains.

\begin{thm}
Let $D\subset\mbb{C}^n$ be a bounded homogeneous domain. Let $\varphi$ be an
automorphism of $D$ such that the group $\Gamma=\langle\varphi\rangle$ is
a discrete subgroup of $\Aut_\mathcal{O}(D)$. Then the quotient $X=D/\Gamma$ is
a Stein space.
\end{thm}

The main steps of the proof are as follows. Since the group
$\Aut_\mathcal{O}(D)$ has only finitely many connected components, we may assume
that $\varphi$ is contained in $G=\Aut_\mathcal{O}(D)^0$. By Kaneyuki's theorem
the group $G$ is isomorphic to the identity component of a real-algebraic group.
Hence, every element $\varphi\in G$ may be written as
$\varphi=\varphi_{\sf e}\varphi_{\sf h}\varphi_{\sf u}$ where $\varphi_{\sf e}$
is elliptic, $\varphi_{\sf h}$ is hyperbolic, $\varphi_{\sf u}$ is unipotent and
where these elements commute. It can be shown that the group
$\Gamma':=\langle\varphi_{\sf h}\varphi_{\sf u}\rangle$ is again discrete in
$G$. Since the groups $\Gamma$ and $\Gamma'$ differ by the compact torus
generated by $\varphi_{\sf e}$, the quotient $X'=D/\Gamma'$ is Stein if and only
if $X$ is Stein. Consequently we may work with the group $\Gamma'$ which has the
advantage of being contained in a maximal split solvable subgroup $S$ of $G$
which acts simply transitively on $D$. Exploiting the structure theory of
$S$ we obtain the existence of an $S$--equivariant holomorphic submersion
$\pi\colon D\to D'$ onto a bounded homogeneous domain $D'$ whose fibers are
biholomorphically equivalent to the unit ball $\mbb{B}_m$. If $\Gamma'$ acts
properly on $D'$ we are in position to use an inductive argument to prove
Steinness of $X$ while if $\Gamma'$ stabilizes every $\pi$--fiber we use the
fact that the quotients $\mbb{B}_m/\Gamma'$ are already known to be Stein.

This paper is organized as follows. In the first section we provide the
necessary background on bounded homogeneous domains and their automorphism
groups. In the second section we establish the existence of a Jordan-Chevalley
decomposition in $G$ and reduce the problem to discrete subgroups of $S$. In
Section~3 we study in detail the unit ball $\mbb{B}_n$ and obtain a new proof
of the fact that $\mbb{B}_n/\Gamma$ is Stein. Afterwards we prove the existence
of the $S$--equivariant submersion $\pi\colon D\to D'$ which allows us to prove
the main result in the last section.

\begin{acknowledgement}
I would like to thank Prof.~Dr.~K.~Oeljeklaus for many helpful and encouraging
discussions on the topics presented here as well as for several invitations to
the Universit{\'e} de Provence (Aix-Marseille~I) where this paper has been
written.
\end{acknowledgement}

\section{Background on bounded homogeneous domains}

We review several facts from the theory of bounded homogeneous domains.
For further details we refer the reader to~\cite{Pya} and~\cite{Kan} and the
references therein.

\subsection{The automorphism group of a bounded homogeneous domain}

Let $D\subset\mbb{C}^n$ be a bounded domain. A theorem of H.~Cartan
(\cite{Car}) states that the group $\Aut_{\mathcal{O}}(D)$ of holomorphic
automorphisms of $D$ is a real Lie group with respect to the compact open
topology such that its natural action on $D$ is differentiable and proper. We
write $G$ for the connected component of $\Aut_{\mathcal{O}}(D)$ which contains
the identity. We identify the Lie algebra $\lie{g}=\Lie(G)$ with the Lie algebra
of complete holomorphic vector fields on $D$.

\begin{definition}
The bounded domain $D$ is called homogeneous if $\Aut_{\mathcal{O}}(D)$ acts
transitively on it.
\end{definition}

\begin{remark}
\begin{enumerate}
\item Let $D$ be a bounded homogeneous domain and let $z_0\in D$ be a base
point. Since $D\cong\Aut_\mathcal{O}(D)/\Aut_\mathcal{O}(D)_{z_0}$ is connected,
the (compact) isotropy group $\Aut_\mathcal{O}(D)_{z_0}$ meets every connected
component of $\Aut_\mathcal{O}(D)$. This shows that $\Aut_\mathcal{O}(D)$ has
at most finitely many connected components.
\item If $D$ is homogeneous, then $G=\Aut_\mathcal{O}(D)^0$ acts transitively
on $D$, too.
\end{enumerate}
\end{remark}

From now on we assume that the bounded domain $D\subset\mbb{C}^n$ is
homogeneous. It follows from~\cite{Bo3} that the group $G$ is semi-simple (and
then in particular real-algebraic) if and only if $D$ is symmetric. For
arbitrary homogeneous domains the group $G$ is semi-algebraic by Kaneyuki's
theorem.

\begin{theorem}[\cite{Kan2}]\label{Thm:Kaneyuki}
There exists a faithful representation $\rho$ of $G$ such that $\rho(G)\subset
{\rm{GL}}(N,\mbb{R})$ is the identity component of a real-algebraic group. In
particular, $\lie{g}$ is isomorphic to an algebraic Lie algebra.
\end{theorem}

Recall that a real Lie algebra $\lie{s}$ is called split solvable if it is
solvable and if the eigenvalues of $\ad(\xi)$ are real for every
$\xi\in\lie{s}$. A Lie group is called split solvable if it is simply-connected
and if its Lie algebra is split solvable. If $G$ is semi-simple, the Iwasawa 
decomposition $K\times A\times N\to G$ exhibits $G$ as diffeomorphic to the
product of its maximal compact subgroup $K$ and its maximal split solvable
subgroup $S:=AN\cong A\ltimes N$. The following theorem of Vinberg generalizes
this decomposition to the group $G=\Aut_\mathcal{O}(D)^0$ for arbitrary bounded
homogeneous domains $D$.

\begin{theorem}[\cite{Vin4}]\label{Thm:Vinberg}
Let $H$ be the connected component of a real-algebraic group. Then there exist
a maximal compact subgroup $K$ and a maximal split solvable subgroup $S$ of $H$
such that the map $K\times S\to H$, $(k,s)\mapsto ks$, is a diffeomorphism. Each
maximal split solvable subgroup of $H$ is conjugate to $S$ by an inner
automorphism of $H$.
\end{theorem}

\begin{remark}
Let $K\times S\to G$ be the decomposition of $G$ from
Theorem~\ref{Thm:Vinberg}. Then $S$ acts simply transitively on $D$.
\end{remark}

For later use we collect some properties of split solvable Lie groups.

\begin{theorem}
Let $S$ be a split solvable Lie group.
\begin{enumerate}
\item The group $S$ is isomorphic to a closed subgroup of the group of upper
triangular matrices in ${\rm{GL}}(N,\mbb{R})$.
\item The exponential map $\exp\colon\lie{s}\to S$ is a diffeomorphism.
\item Every connected subgroup of $S$ is closed and simply-connected.
\item For each element $g\in S$ the group $\{g^k;\ k\in\mbb{Z}\}$ is a discrete
subgroup of $S$ isomorphic to $\mbb{Z}$.
\item Let $S'\subset S$ be a connected subgroup and let $(S')^\mbb{C}\subset
S^\mbb{C}$ be their universal complexifications in the sense of~\cite{Ho}. Then
the homogeneous space $S^\mbb{C}/(S')^\mbb{C}$ is biholomorphic to
$\mbb{C}^{\dim S-\dim S'}$.
\end{enumerate}
\end{theorem}

\begin{proof}
The first three statements are classical (see for example~\cite{Vin3}). The
fourth assertion is a direct consequence of the second one. A proof of the last
assertion can be found in~\cite{HuOe}.
\end{proof}

\subsection{Siegel domains and the grading of $\lie{g}$}

In this subsection we will describe the notion of Siegel domains of the first
and of the second kind. Our motivation for the study of these domains comes
from the fact that each bounded homogeneous domain can be realized as a Siegel
domain (\cite{GinPyaVin}). In addition we discuss the grading of $\lie{g}$ which
has been introduced in~\cite{KaMatOch}.

Let $V$ be a finite-dimensional real vector space and let $\Omega\subset V$ be a
regular cone, i.\,e.\ an open convex cone which does not contain any affine
line.

\begin{definition}
The tube domain $D:=D(\Omega):=\bigl\{z\in V^\mbb{C};\ \im(z)\in\Omega\bigr\}
=V+i\Omega$ is called the Siegel domain of the first kind associated with
$\Omega$.
\end{definition}

\begin{remark}
The assumption that $D$ is a tube domain over a regular cone is quite strong.
Although the unit ball in $\mbb{C}^n$ is biholomorphically equivalent to a tube
domain over a convex domain in $\mbb{R}^n$, it can not be realized as a Siegel
domain of the first kind.
\end{remark}

The automorphism group $G(\Omega)$ of $\Omega$ is defined by
\begin{equation*}
G(\Omega):=\bigl\{g\in{\rm{GL}}(V);\ g(\Omega)=\Omega\bigr\}.
\end{equation*}
Since the condition $g(\Omega)=\Omega$ is equivalent to $g(\overline{\Omega})=
\overline{\Omega}$, the group $G(\Omega)$ is closed in ${\rm{GL}}(V)$ and hence
a Lie group. We embed $G(\Omega)$ into the automorphism group of $D=D(\Omega)$
by $g\mapsto\varphi_g$ with $\varphi_g(z)=gz$.

Let $W$ be a finite-dimensional complex vector space. A map $\Phi\colon W\times
W\to V^\mbb{C}$ is called $\Omega$--Hermitian if the following holds:
\begin{enumerate}
\item For all $w'\in W$ the map $w\mapsto\Phi(w,w')$ is complex-linear.
\item We have $\Phi(w',w)=\overline{\Phi(w,w')}$ for all $w,w'\in W$.
\item We have $\Phi(w,w)\in\overline{\Omega}$ for all $w\in W$, and
$\Phi(w,w)=0$ if and only if $w=0$.
\end{enumerate}

\begin{remark}
If $V=\mbb{R}$ and $\Omega=\mbb{R}^{>0}$, then an $\Omega$--Hermitian form is
the same as a positive definite Hermitian form on $W$.
\end{remark}

\begin{definition}
Given $\Omega$ and $\Phi$ as above, the domain
\begin{equation*}
D:=D(\Omega,\Phi):=\bigl\{(z,w)\in V^\mbb{C}\times W;\ \im(z)-\Phi(w,w)
\in\Omega\bigr\}
\end{equation*}
is called the Siegel domain of the second kind associated to $\Omega$ and
$\Phi$.
\end{definition}

\begin{proposition}\label{Prop:SiegelDomain}
Every Siegel domain of the first or second kind is convex and biholomorphically
equivalent to a bounded domain. Hence, each Siegel domain $D$ is a domain of
holomorphy and its automorphism group is a real Lie group acting properly on
$D$.
\end{proposition}

\begin{proof}
Convexity of Siegel domains is elementary to check. For a proof of the fact that
$D$ is biholomorphically equivalent to a bounded domain we refer the reader
to~\cite{Pya}.
\end{proof}

\begin{theorem}[\cite{GinPyaVin}]
Every bounded homogeneous domain can be realized as a Siegel domain of either
the first or the second kind.
\end{theorem}

Let $D=D(\Omega,\Phi)$ be a Siegel domain. As usual we write $G$ for the
connected component of the identity in $\Aut_\mathcal{O}(D)$. Let us introduce
linear coordinates $z_k$, $1\leq k\leq\dim_\mbb{C}V^\mbb{C}$, in $V^\mbb{C}$ and
$w_\alpha$, $1\leq\alpha\leq\dim_\mbb{C}W$, in $W$. It follows from the
definition that $\lie{g}$ contains the vector field
\begin{equation*}
\delta:=\sum_kz_k\vf{z_k}+\frac{1}{2}\sum_\alpha w_\alpha\vf{w_\alpha}.
\end{equation*}

\begin{theorem}[\cite{KaMatOch}]\label{Thm:Grading}
The Lie algebra $\lie{g}$ admits a decomposition
\begin{equation*}
\lie{g}=\lie{g}_{-1}\oplus\lie{g}_{-1/2}\oplus\lie{g}_0\oplus\lie{g}_{1/2}
\oplus\lie{g}_1,
\end{equation*}
where $\lie{g}_\lambda$ is the eigenspace of $\ad(\delta)$ for the eigenvalue
$\lambda$. Then the following holds.
\begin{enumerate}
\item We have $[\lie{g}_\lambda,\lie{g}_\mu]\subset\lie{g}_{\lambda+\mu}$ for
all $\lambda,\mu\in\{\pm1,\pm1/2,0\}$.
\item The translation vector fields $\vf{z_k}$, $1\leq k\leq
\dim_\mbb{C}V^\mbb{C}$, form a basis of $\lie{g}_{-1}$. Consequently, we have
$\dim\lie{g}_{-1}=\dim_\mbb{C}V^\mbb{C}$.
\item The elements of $\lie{g}_{-1/2}$ are of the form
\begin{equation*}
2i\sum_k\Phi_k(w,c)\vf{z_k}+\sum_\alpha c_\alpha\vf{w_\alpha}
\qquad(c\in\mbb{C}^{\dim_\mbb{C}W}).
\end{equation*}
Consequently, $\dim\lie{g}_{-1/2}=2\dim_\mbb{C}W$, and $\lie{g}_{-1/2}=\{0\}$ if
and only if $D$ is a Siegel domain of the first kind.
\item The Lie subalgebra $\lie{g}_0$ consists of all elements of the form
\begin{equation*}
\sum_{k,l}a_{kl}z_k\vf{z_{l}}+\sum_{\alpha,\beta}b_{\alpha\beta}w_\alpha
\vf{w_\beta},
\end{equation*}
where the matrix $A:=(a_{kl})$ lies in the Lie algebra of $G(\Omega)$ and
$B:=(b_{\alpha\beta})\in\lie{gl}(W)$ fulfills
\begin{equation*}
A\Phi(w,w')=\Phi(Bw,w')+\Phi(w,Bw')
\end{equation*}
for all $w,w'\in W$.
\item The subalgebra $\lie{g}_{-1}\oplus\lie{g}_{-1/2}\oplus\lie{g}_0$ is the
Lie algebra of the group of affine automorphisms of $D$.
\end{enumerate}
\end{theorem}

Theorem~\ref{Thm:Grading} allows us to find a particularly nice maximal split
solvable subalgebra $\lie{s}$ of $\lie{g}$.

\begin{proposition}
Let $\lie{s}_0$ be a maximal split solvable subalgebra of $\lie{g}_0$. Then
$\lie{s}:=\lie{g}_{-1}\oplus\lie{g}_{-1/2}\oplus\lie{s}_0$ is a maximal split
solvable subalgebra of $\lie{g}$.
\end{proposition}

\begin{proof}
This is the content of Proposition~2.8 in~\cite{Kan}.
\end{proof}

\subsection{Normal $j$--algebras}

We have seen that every bounded homogeneous domain $D$ is diffeomorphic to a
split solvable Lie algebra $\lie{s}$. Transferring the complex structure and the
Bergman metric of $D$ to $\lie{s}$ we obtain the notion of a normal
$j$--algebra which was introduced by Pyateskii-Shapiro. We follow the
exposition in~\cite{Is}. Complete proofs and further details can be found
in~\cite{Pya}.

\begin{definition}
A normal $j$--algebra is a pair $(\lie{s},j)$ of a split solvable Lie
algebra $\lie{s}$ and a complex structure $j$ on $\lie{s}$ such that
\begin{equation}\label{Eqn:Integrability}
[\xi,\xi']+j[j\xi,\xi']+j[\xi,j\xi']-[j\xi,j\xi']=0
\end{equation}
for all $\xi,\xi'\in\lie{s}$. In addition, we demand the existence of a linear
form $\omega\in\lie{s}^*$ such that
\begin{equation*}
\langle\xi,\xi'\rangle_\omega:=\omega\bigl([j\xi,\xi']\bigr)
\end{equation*}
defines a $j$--invariant inner product on $\lie{s}$.
\end{definition}

\begin{remark}
If we extend the complex structure $j$ on $\lie{s}$ to a left invariant complex
structure $J$ on the simply-connected group $S$, then
condition~\eqref{Eqn:Integrability} guarantees that $S$ is a complex manifold
with respect to $J$.
\end{remark}

Let us describe the fine structure of a normal $j$--algebra $(\lie{s},j)$ via a
root space decomposition. Since $\lie{s}$ is solvable, its derived algebra
$\lie{n}:=[\lie{s},\lie{s}]$ is nilpotent. Let $\lie{a}$ denote the orthogonal
complement of $\lie{n}$ with respect to $\langle\cdot,\cdot\rangle_\omega$.
Hence, we obtain $\lie{s}=\lie{a}\oplus\lie{n}$ and one can show that $\lie{a}$
is a maximal Abelian subalgebra consisting of semi-simple elements of
$\lie{s}$. The dimension $r:=\dim\lie{a}$ is called the rank of $\lie{s}$. Since
$\lie{s}$ is split solvable, we can form the root space decomposition
\begin{equation}\label{Eqn:RootSpaceDecomp}
\lie{s}=\lie{a}\oplus\bigoplus_{\alpha\in\Delta}\lie{s}_\alpha,
\end{equation}
where we write $\lie{s}_\alpha:=\bigl\{\xi\in\lie{s};\ [\eta,\xi]=\alpha(\eta)
\xi\bigr\}$ for $\alpha\in\lie{a}^*$ and $\Delta:=\Delta(\lie{s},\lie{a}):=
\bigl\{\alpha\in\lie{a}^*\setminus\{0\};\ \lie{s}_\alpha\not=\{0\}\bigr\}$.

\begin{proposition}\label{Prop:Finestructure}
Let $(\lie{s},j)$ be a normal $j$--algebra.
\begin{enumerate}
\item The root space decomposition~\eqref{Eqn:RootSpaceDecomp} is orthogonal
with respect to $\langle\cdot,\cdot\rangle_\omega$.
\item There exist $r$ linearly independent roots $\alpha_1,\dotsc,\alpha_r$
such that all other roots are of the form
\begin{equation*}
\tfrac{1}{2}\alpha_k\ (1\leq k\leq r)\quad\text{and}\quad\tfrac{1}{2}
(\alpha_l\pm\alpha_k)\ (1\leq k<l\leq r).
\end{equation*}
Note that not all possibilities have to occur.
\item Let $(\eta_1,\dotsc,\eta_r)$ be the basis of $\lie{a}$ dual to
$(-\alpha_1,\dotsc,-\alpha_r)$ and set $\xi_k:=-j\eta_k$. Then we have
$\lie{s}_{\alpha_k}=\mbb{R}\xi_k$ for all $1\leq k\leq r$.
\item For all $1\leq k<l\leq r$ we have $j\lie{s}_{(\alpha_l-\alpha_k)/2}
=\lie{s}_{(\alpha_l+\alpha_k)/2}$.
\item For all $1\leq k\leq r$ we have $j\lie{s}_{\alpha_k/2}=
\lie{s}_{\alpha_k/2}$.
\end{enumerate}
\end{proposition}

Finally, we set $\delta:=\eta_1+\dotsb+\eta_r$ and write $\lie{s}_\lambda$ for
the eigenspace of $\ad(\delta)$ with eigenvalue $\lambda\in\mbb{R}$. Then we
obtain the grading
\begin{equation*}
\lie{s}=\lie{s}_{-1}\oplus\lie{s}_{-1/2}\oplus\lie{s}_0
\end{equation*}
of $\lie{s}$ where
\begin{equation}
\begin{split}
\lie{s}_{-1}&=\bigoplus_{k=1}^r\lie{s}_{\alpha_k}\oplus\bigoplus_{1\leq k<l
\leq r}\lie{s}_{(\alpha_l+\alpha_k)/2},\\
\lie{s}_{-1/2}&=\bigoplus_{k=1}^r\lie{s}_{\alpha_k/2},\\
\lie{s}_0&=\lie{a}\oplus\bigoplus_{1\leq k<l\leq r}
\lie{s}_{(\alpha_l-\alpha_k)/2}
\end{split}
\end{equation}
hold.

Next we explain how the domain $D$ can be recovered from $(\lie{s},j)$. Let $S$
be the simply-connected Lie group with Lie algebra $\lie{s}$ and let $S_0$ be
the analytic subgroup whose Lie algebra is given by $\lie{s}_0$. We define
$\xi:=\xi_1+\dotsb+\xi_r$ and $\Omega:=\Ad(S_0)\xi$. One can show that $\Omega$
is a regular cone in $\lie{s}_{-1}$.

Since $\lie{s}_{-1/2}$ is invariant under $j$, we may consider
$(\lie{s}_{-1/2},j)$ as a complex vector space. Then the map $\Phi\colon
\lie{s}_{-1/2}\times\lie{s}_{-1/2}\to\lie{s}_{-1}^\mbb{C}$,
\begin{equation*}
\Phi(\xi,\xi'):=\frac{1}{4}\bigl([j\xi,\xi']+i[\xi,\xi']\bigr)
\end{equation*}
is an $\Omega$--Hermitian form on $\lie{s}_{-1/2}$. Hence, we obtain the
associated Siegel domain
\begin{equation*}
D_\lie{s}:=\bigl\{(\xi,\xi')\in\lie{s}_{-1}^\mbb{C}\times\lie{s}_{-1/2};\
\im(\xi)-\Phi(\xi',\xi')\in\Omega\bigr\}.
\end{equation*}

Let $S_{-}$ be the analytic subgroup of $S$ corresponding to $\lie{s}_{-1}\oplus
\lie{s}_{-1/2}$. Then the group $S=S_{-}\rtimes S_0$ acts by affine maps on
$D_\lie{s}$ via
\begin{equation}\label{Eqn:Action}
\bigl(\exp(\xi+\xi'),s\bigr)\cdot(z,w)\\:=\Bigl(\Ad(s)z+\xi+2i
\Phi\bigl(\Ad(s)w,\xi'\bigr)+i \Phi(\xi',\xi'),\Ad(s)w+\xi'\Bigr),
\end{equation}
where $\xi\in\lie{s}_{-1}$, $\xi'\in\lie{s}_{-1/2}$, $s\in S_0$ and $(z,w)\in
\lie{s}_{-1}^\mbb{C}\times\lie{s}_{-1/2}$ hold. One can show that this action is
simply transitive on $D_\lie{s}$ which implies that $D_\lie{s}$ is
biholomorphically equivalent to a bounded homogeneous domain.

\begin{theorem}
The construction described above yields a one-to-one correspondence between
equivalence classes of homogeneous bounded domains and isomorphism classes of
normal $j$--algebras.
\end{theorem}

Finally we note the following corollary of~\eqref{Eqn:Action}.

\begin{lemma}\label{Lem:TransitiveAction}
The group $S^\mbb{C}$ acts transitively on $\lie{s}_{-1}^\mbb{C}\times
\lie{s}_{-1/2}$.
\end{lemma}

\section{Cyclic groups acting on bounded homogeneous
domains}\label{Sec:Reduction}

We carry out the first step towards a proof of Steinness of $X=D/\Gamma$ by
showing that it is enough to assume that the cyclic group $\Gamma$ lies in a
maximal split solvable subgroup of $G$.

\subsection{Reduction to automorphisms in $G$}

Let $D\subset\mbb{C}^n$ be a bounded homogeneous domain and let $\varphi\in
\Aut_\mathcal{O}(D)$ be such that the subgroup $\Gamma:=\langle\varphi\rangle:=
\{\varphi^m;\ m\in\mbb{Z}\}$ is discrete in $\Aut_\mathcal{O}(D)$. Since every
discrete subgroup is also closed, this implies that $\Gamma$ acts properly on
$D$, and hence that $X:=D/\Gamma$ is a complex space.

Since the group $\Gamma$ is cyclic, it is either finite or isomorphic to
$\mbb{Z}$. In the first case it is classical that Steinness of $D$ implies
Steinness of $X$ (see for example~\cite{GrRe2}). Therefore we will assume in
the following that $\Gamma$ is isomorphic to $\mbb{Z}$. Since every proper
$\mbb{Z}$--action is automatically free, the quotient $X$ is a complex manifold
in this case.

Recall that the group $\Aut_\mathcal{O}(D)$ has only finitely many connected
components which implies that $\Gamma^0:=\Gamma\cap G$ is a normal subgroup of
finite index in $\Gamma$. Since $D/\Gamma$ is Stein if and only if $D/\Gamma^0$
is so, we may assume without loss of generality that $\varphi$ is contained in
$G=\Aut_\mathcal{O}(D)^0$.

\subsection{Jordan-Chevalley decomposition}

In this subsection we will explain how Kaneyuki's Theorem~\ref{Thm:Kaneyuki}
implies the existence of the Jordan-Chevalley decomposition in $G$.

Let us quickly review the Jordan-Chevalley decomposition. If $H$ is a
real-algebraic group, then every element $h\in H$ can be uniquely written as
$h=h_{\sf s}h_{\sf u}=h_{\sf u}h_{\sf s}$ where $h_{\sf s}\in H$ is semi-simple
and $h_{\sf u}\in H$ is unipotent. Following~\cite{Kos} we decompose the
semi-simple part $h_{\sf s}$ further as $h_{\sf s}=h_{\sf e}h_{\sf h}$ where the
eigenvalues of $h_{\sf e}\in H$ lie in the unit circle in $\mbb{C}$ and where
$h_{\sf h}\in H$ has only positive real eigenvalues. We call $h_{\sf e}$ the
elliptic and $h_{\sf h}$ the hyperbolic part of $h$. Note that the elements
$h_{\sf e}$, $h_{\sf h}$ and $h_{\sf u}$ commute.

\begin{lemma}\label{Lem:JordanDecomp}
Let $H\subset{\rm{GL}}(N,\mbb{R})$ be a real-algebraic group and let $h\in H^0$
be given. If $h=h_{\sf e}h_{\sf h}h_{\sf u}$ is the multiplicative Jordan
decomposition of $h$ in $H$, then we have $h_{\sf e},h_{\sf h},h_{\sf u}\in
H^0$.
\end{lemma}

\begin{proof}
Let $h=h_{\sf s}+h_{\sf n}$ be the additive Jordan decomposition in
$\mbb{R}^{N\times N}$. As is well known the matrices $h_{\sf s}$ and $h_{\sf n}$
can be expressed as polynomials in $h$. Furthermore, the multiplicative Jordan
decomposition of $h$ is then given by $h=h_{\sf s}h_{\sf u}$ with
$h_{\sf u}=I_N+h_{\sf s}^{-1}h_{\sf n}$. Since $H$ is real-algebraic, we have
$h_{\sf s},h_{\sf u}\in H$ and the matrices $h_{\sf s}$ and $h_{\sf u}$ depend
continuously on $h$. Moreover, the matrices $h_{\sf e}$ and $h_{\sf h}$ lie in
$H$ and depend continuously on $h$, too.

If $h\in H^0$ holds, we find a continuous curve $t\mapsto h(t)\in H^0$,
$t\in[0,1]$, with $h(0)=I_N$ and $h(1)=h$. Forming the multiplicative Jordan
decomposition $h(t)=h_{\sf e}(t)h_{\sf h}(t)h_{\sf u}(t)$ we obtain continuous
curves $t\mapsto h_{\sf e}(t)\in H$, $t\mapsto h_{\sf h}(t)\in H$ and $t\mapsto
h_{\sf u}(t)\in H$. Because of $h_{\sf e}(0)=h_{\sf h}(0)=h_{\sf u}(0)=I_N$ the
claim follows.
\end{proof}

Since Kaneyuki's Theorem asserts that there exists a faithful representation
$\rho\colon G\to{\rm{GL}}(N,\mbb{R})$ such that $\rho(G)=H^0$ for a
real-algebraic subgroup $H\subset{\rm{GL}}(N,\mbb{R})$, we obtain the following
notion of Jordan-Chevalley decomposition in $G$.

\begin{definition}
We say that an element $g\in G$ is elliptic, hyperbolic or unipotent if the
element $\rho(g)$ has this property.
\end{definition}

\begin{proposition}
Every element $g\in G$ may be uniquely written as $g=g_{\sf e}g_{\sf h}
g_{\sf u}$ where $g_{\sf e}$ is elliptic, $g_{\sf h}$ is hyperbolic and
$g_{\sf u}$ is unipotent and where these three elements commute with each other.
\end{proposition}

The following proposition generalizes Propositions~2.3 and~2.5 of~\cite{Kos}.

\begin{proposition}\label{Prop:Conjugacy}
Every elliptic element of $G$ is conjugate to an element in the maximal compact
subgroup $K$, while every element $g\in G$ with $g_{\sf e}=e$ is conjugate to an
element in the maximal split solvable group $S$.
\end{proposition}

\begin{proof}
The claim follows from the facts that elliptic elements generate compact
groups, that elements with trivial elliptic part generate split solvable groups
and that maximal compact respectively split solvable groups are conjugate.
\end{proof}

\subsection{Reduction to automorphisms with trivial elliptic part}

In this subsection we will show that it is enough to consider automorphisms
$\varphi\in G$ whose elliptic part vanishes.

Let $\varphi=\varphi_{\sf e}\varphi_{\sf h}\varphi_{\sf u}$ be the
Jordan-Chevalley decomposition of $\varphi$ and set
$\varphi':=\varphi_{\sf h}\varphi_{\sf u}$ as well as
$\Gamma':=\langle\varphi'\rangle$. By Proposition~\ref{Prop:Conjugacy} we may
assume that the group $\Gamma'$ is contained in the split solvable subgroup $S$
of $G$. This implies in particular that $\Gamma'$ is a closed subgroup of $G$.
Thus we may consider the complex manifold $X':=D/\Gamma'$. We will show that $X$
is Stein if and only if $X'$ is Stein.

The closure $T$ of the group generated by $\varphi_e$ is a compact torus in
$G$. Since $\varphi'$ and $\varphi_e$ commute, we conclude that $\Gamma$ and
$\Gamma'$ lie in the centralizer $\mathcal{Z}_G(T)$. Consequently, the
sets $T\Gamma$ and $T\Gamma'$ are subgroups of $G$.

\begin{lemma}
We have $T\Gamma=T\Gamma'$, and the action of $T\Gamma'$ on $D$ is proper.
Hence, $Y:=D/(T\Gamma)=D/(T\Gamma')$ is a Hausdorff topological space.
\end{lemma}

\begin{proof}
The identity $T\Gamma=T\Gamma'$ is elementary to check.

Since $T$ is compact, the $T$--action on $X'$ is proper, hence the product
group $T\times\Gamma'$ acts properly on $D$. Since the element $\varphi'$ has
by definition trivial elliptic part, the group $T\Gamma'$ is isomorphic to $T
\times\Gamma'$.
\end{proof}

Since the groups $\Gamma$ and $\Gamma'$ are normal in $T\Gamma=T\Gamma'$, the
torus $T$ acts properly on $X$ and $X'$ and we obtain the following commutative
diagram:
\begin{equation*}
\xymatrix{
 & D\ar[dl]_p\ar[dd]^{\pi}\ar[dr]^{p'} & \\
X\ar[dr]_q & & X'\ar[dl]^{q'}\\
 & Y. &
}
\end{equation*}

The following proposition is the main result of this subsection.

\begin{proposition}\label{Prop:Reduction}
The manifold $X$ is Stein if and only if $X'$ is Stein. Hence, we can restrict
our attention to automorphisms with trivial elliptic part.
\end{proposition}

\begin{proof}
In a first step we investigate how $T$--invariant functions on $X'$ induce
$T$--invariant functions on $X$. For this let $f\colon X'\to\mbb{R}$ be any
smooth function which is invariant under $T$. It follows that the pull-back
$(p')^*f\colon D\to\mbb{R}$ is smooth and $T\Gamma'$--invariant. Since $\Gamma$
is a normal subgroup of $T\Gamma'$, we obtain a $T$--invariant smooth function
$\widetilde{f}\colon X\to\mbb{R}$.

Since the above diagram commutes, $f$ and $\widetilde{f}$ induce the same
continuous function on $Y$. By compactness of $T$ this implies that if $f$ is an
exhaustion, then $\widetilde{f}$ is also an exhaustion. Moreover, if $f$ is
strictly plurisubharmonic, then $(p')^*f$ is strictly plurisubharmonic and hence
$\widetilde{f}$ is strictly plurisubharmonic.

If $X'$ is Stein, then there exists a strictly plurisubharmonic exhaustion
function on $X'$. Since $T$ is compact, we can assume that this function is
$T$--invariant. By the above arguments, we obtain a strictly plurisubharmonic
exhaustion function on $X$. Hence, $X$ is Stein.

The converse is proved similarly.
\end{proof}

\section{Example: The unit ball in $\mbb{C}^n$}

In this section we discuss the automorphism group and the normal $j$--algebra of
the unit ball $\mbb{B}_n:=\{z\in\mbb{C}^n;\ \norm{z}<1\}$ in $\mbb{C}^n$. It has
been proven in~\cite{Fab} and~\cite{FabIan} that the quotient manifold
$\mbb{B}_n/\langle\varphi\rangle$ is Stein for hyperbolic and parabolic
automorphisms $\varphi\in\Aut_\mathcal{O}(\mbb{B}_n)$. We will give here a
different proof of this fact.

\subsection{The automorphism group of the unit ball}

Let us first describe the full automorphism group of the unit ball
$\mbb{B}_n\subset\mbb{C}^n$. For this we embed $\mbb{C}^n$ into the complex
projective space $\mbb{P}_n(\mbb{C})$ by $(z_1,\dotsc,z_n)\mapsto[z_1:\dotsb:
z_n:1]$. The image of $\mbb{B}_n$ under this embedding is given by
\begin{equation*}
D:=\bigl\{[z_1:\dotsb:z_{n+1}]\in\mbb{P}_n(\mbb{C});\ \abs{z_1}^2+\dotsb+
\abs{z_n}^2-\abs{z_{n+1}}^2<0\bigr\}.
\end{equation*}
Consequently, the group ${\rm{SU}}(n,1)$, acting as a subgroup of
${\rm{SL}}(n+1,\mbb{C})$ by projective transformations on $\mbb{P}_n(\mbb{C})$,
leaves $D$ invariant. Hence, we obtain a homomorphism
$\Phi\colon{\rm{SU}}(n,1)\to\Aut_\mathcal{O}(\mbb{B}_n)$. One can show that
$\Phi$ is a surjective homomorphism of Lie groups whose kernel coincides with
the (finite) center of ${\rm{SU}}(n,1)$ (see for example~\cite{Akh}).

In order to find explicit formulas for the automorphisms of $\mbb{B}_n$
belonging to a maximal split solvable subgroup $B_n$ of $G=\Aut_\mathcal{O}(
\mbb{B}_n)$ we make use of the realization of $\mbb{B}_n$ as the Siegel domain
\begin{equation*}
\mbb{H}_n:=\bigl\{(z,w)\in\mbb{C}\times\mbb{C}^{n-1};\
\im(z)-\norm{w}^2>0\bigr\}.
\end{equation*}
From Theorem~\ref{Thm:Grading} we obtain $2n$ one parameter subgroups of
automorphisms of $\mbb{H}_n$ which generate the group $B_n$. These are listed
together with their corresponding complete holomorphic vector fields in
Table~\ref{Table:AutomUnitBall}.

\begin{table}
\begin{center}
\begin{tabular}{cc}\toprule
One-parameter group of automorphisms &  Vector field\\ \midrule
$(z,w)\mapsto(z+t,w)$ &   $\zeta=\vf{z}$\\\addlinespace
$(z,w)\mapsto(z+2itw_k+it^2,w_1,\dotsc,w_k+t,\dotsc,w_{n-1})$ &
$\xi_k=2iw_k\vf{z}+\vf{w_k}$ ($1\leq k\leq n-1$)\\ \addlinespace
$(z,w)\mapsto(z+2tw_k+it^2,w_1,\dotsc,w_k+it,\dotsc,w_{n-1})$ &
$\eta_k=2w_k\vf{z}+i\vf{w_k}$ ($1\leq k\leq n-1$)\\ \addlinespace
$(z,w)\mapsto(e^tz,e^{t/2}w)$ &
$\delta=z\vf{z}+\frac{1}{2}\sum_{k=1}^{n-1}w_k\vf{w_k}$\\\bottomrule
\end{tabular}
\end{center}
\caption{Automorphisms of $\mbb{H}_n$ generating $B_n\cong A\ltimes
N_n$}\label{Table:AutomUnitBall}
\end{table}

\subsection{The normal $j$--algebra of the unit ball}

Let $\lie{b}_n$ be the Lie algebra of the group $B_n$. Its derived algebra
$\lie{n}_n:=[\lie{b}_n,\lie{b}_n]$ is given by
\begin{equation*}
\lie{n}_n=\mbb{R}\zeta\oplus\mbb{R}\xi_1\oplus\dotsb\oplus\mbb{R}\xi_{n-1}
\oplus\mbb{R}\eta_1\oplus\dotsb\oplus\mbb{R}\eta_{n-1},
\end{equation*}
while $\lie{a}:=\mbb{R}\delta$ is maximal Abelian consisting of semi-simple
elements of $\lie{b}_n$. One computes directly that the only non-vanishing
commutators are
\begin{equation*}
[\delta,\xi_k]=-\tfrac{1}{2}\xi_k,\quad[\delta,\eta_k]=-\tfrac{1}{2}\eta_k,
\quad[\delta,\zeta]=-\zeta,\quad[\xi_k,\eta_k]=4\zeta.
\end{equation*}
In particular, $\lie{n}_n$ is a $(2n-1)$--dimensional Heisenberg algebra with
center $\mbb{R}\zeta$. Choosing the base point $z_0=(i,0)\in\mbb{H}_n$ we
obtain via the isomorphism $\lie{b}_n\to\lie{b}_n\cdot z_0=T_{z_0}\mbb{H}_n$
the following complex structure $j$ on $\lie{b}_n$:
\begin{equation*}
j\zeta=\delta,\quad j\xi_k=\eta_k.
\end{equation*}
These data describe the normal $j$--algebra $(\lie{b}_n,j)$ of the unit ball
$\mbb{B}_n$.

In the rest of this subsection we will prove several technical facts which lead
to a proof of Steinness of $\mbb{B}_n/\Gamma$.

\begin{lemma}\label{Lem:TubeRealization}
Let $\xi\in\lie{n}_n$ be arbitrary. Then there exists an $n$--dimensional
Abelian subalgebra $\lie{n}'_n$ of $\lie{n}_n$ which contains $\xi$.
\end{lemma}

\begin{proof}
We proof the claim by induction over $n$. For $n=1$ the subalgebra $\lie{n}_1$
itself is one-dimensional and Abelian. Hence, let $n>1$ and let us assume that
the claim holds for $n-1$. We write $\xi=\xi'+\xi''$ according to the
decomposition $\lie{n}_n=\lie{n}_{n-1}\oplus\mbb{R}\xi_{n-1}\oplus
\mbb{R}\eta_{n-1}$. By our induction hypotheses there exists an
$(n-1)$--dimensional Abelian subalgebra $\lie{n}_{n-1}'$ of $\lie{n}_{n-1}$
containing $\xi'$. Then $\lie{n}_n':=\lie{n}_{n-1}'\oplus\mbb{R}\xi''$ has the
required properties.
\end{proof}

As a consequence we obtain the following

\begin{proposition}[\cite{Ya},\cite{Ian}]\label{Prop:TubeRealization}
Let $N_n'$ be the analytic subgroup of $B_n$ with Lie algebra $\lie{n}'_n$.
Then every $N'_n$--orbit in $\mbb{B}_n$ is totally real and $\mbb{B}_n$ is
biholomorphically equivalent to a tube domain $D$ in $\mbb{C}^n$ such that
$N'_n$ acts by translations on $D$.
\end{proposition}

For the proof we have to review parts of the theory of (universal)
globalizations of local holomorphic actions. We use~\cite{HeIan} as a general
reference.

Let $M$ be a complex manifold endowed with a local holomorphic action of a
complex Lie group $L$. A globalization of this local action consists in an open
holomorphic embedding $\iota$ of $M$ into a (possibly non-Hausdorff) complex
manifold $M^*$ on which $L$ acts holomorphically such that $\iota$ is locally
equivariant and $M^*=L\cdot \iota(M)$. A globalization $M^*$ is called universal
if for every locally $L$--equivariant map $\varphi\colon M\to M'$ into an
$L$--manifold $M'$ there exists a unique $L$--equivariant map $\varphi^*\colon
M^*\to M'$ such that
the diagram
\begin{equation*}
\xymatrix{
M\ar[rr]^{\varphi}\ar[dr]_{\iota} & & M'\\
& M^*\ar[ur]_{\varphi^*} &
}
\end{equation*}
commutes. By remark in~\S 3 in~\cite{HeIan} the universal globalization of
a local $L$--action on $M$ exists if and only if any globalization exists.

\begin{proof}[Proof of Proposition~\ref{Prop:TubeRealization}]
Since the group $N_n'$ is Abelian and $\mbb{H}_n$ is hyperbolic, every
$N'_n$--orbit in $\mbb{H}_n$ must be totally real.

Let $(N_n')^\mbb{C}$ be the universal complexification of $N_n'$. Since
$(N'_n)^\mbb{C}$ acts by affine-linear transformations on $\mbb{C}^n$, the
universal globalization $\mbb{H}_n^*$ of the local $(N'_n)^\mbb{C}$--action on
$\mbb{H}_n$ exists. Since every $N'_n$--orbit is totally real and of maximal
dimension $n$, every $(N'_n)^\mbb{C}$--orbit in $\mbb{H}_n^*$ is open. Thus
$\mbb{H}_n^*\cong (N'_n)^\mbb{C}/(N'_n)^\mbb{C}_z$ is homogeneous and in
particular Hausdorff. Moreover, $\mbb{H}_n$ is biholomorphically equivalent to
a $N'_n$--invariant domain in this homogeneous space. Because of
$\dim\mbb{H}_n^*=n=\dim(N'_n)^\mbb{C}$ the isotropy $(N'_n)^\mbb{C}_z$ is
discrete. Since $(N'_n)^\mbb{C}\cong\mbb{C}^n$ is simply-connected, we may apply
Lemma~2.1 of~\cite{IanSpTr} in order to conclude that $\mbb{H}_n^*$ is
simply-connected which implies $(N'_n)^\mbb{C}_z=\{e\}$. Hence the claim
follows.
\end{proof}

\begin{lemma}
Let $\xi=\xi_\lie{a}+\xi_{\lie{n}_n}\in\lie{a}\oplus\lie{n}_n=\lie{b}_n$ be an
element with $\xi_\lie{a}\not=0$. Then there exists an element $g\in B_n$ with
$\Ad(g)\xi\in\lie{a}$.
\end{lemma}

\begin{proof}
We prove the lemma by induction over $n$. If $n=1$, we identify the Lie algebra
$\lie{b}_1$ with $\left\{\left(\begin{smallmatrix}t&s\\0&-t\end{smallmatrix}
\right);\ t,s\in\mbb{R}\right\}$. If $\xi=\left(\begin{smallmatrix}
t&s\\0&-t\end{smallmatrix}\right)$ with $t\not=0$ is given, one verifies that
$g=\left(\begin{smallmatrix}1&s/2t\\0&1\end{smallmatrix}\right)\in B_1$
fulfills $\Ad(g)\xi\in\lie{a}$.

Let $n>1$ and let us assume that the claim is proven for $n-1$. We write
$\xi=\xi'+\xi''$ according to $\lie{b}_n=\lie{b}_{n-1}\oplus(\mbb{R}\xi_{n-1}
\oplus\mbb{R}\eta_{n-1})$. Since $\xi'_\lie{a}=\xi_\lie{a}\not=0$ (and in
particular $\xi'\not=0$), our induction hypothesis implies the existence of an
element $g\in B_{n-1}$ such that $\Ad(g)\xi'\in\lie{a}$ holds. Since we have
\begin{equation*}
[\xi',\xi'']=[\xi'_\lie{a},\xi'']+[\xi'_{\lie{n}_{n-1}},\xi'']=-\lambda\xi''
\end{equation*}
for some $\lambda\not=0$, the subspace
$\mbb{R}\Ad(g)\xi'\oplus\mbb{R}\Ad(g)\xi''\subset\lie{a}\oplus\lie{n}_n$ is a
subalgebra of $\lie{b}_n$ isomorphic to $\lie{b}_1$. Since $\Ad(g)\xi'\not=0$,
there is an element $g'$ in the corresponding subgroup with
$\Ad(g')\bigl(\Ad(g)\xi'+ \Ad(g)\xi''\bigr)=\Ad(g'g)\xi\in\lie{a}$. Hence, the
lemma is proven.
\end{proof}

\begin{lemma}
The subspace $\lie{b}_n':=\lie{a}\oplus\mbb{R}\xi_1\oplus\dotsb\oplus
\mbb{R}\xi_{n-1}$ is an $n$--dimensional subalgebra of $\lie{b}_n$ such that
every orbit of the corresponding subgroup $B'_n$ of $B_n$ is totally real in
$\mbb{H}_n$.
\end{lemma}

\begin{proof}
Using the commutator relations one checks directly that $\lie{b}_n'$ is a
subalgebra of $\lie{b}_n$.

To prove the second claim note that for $(z,w)\in\mbb{C}\times\mbb{C}^{n-1}$ we
have
$T_{(z,w)}\bigl(B'_n\cdot(z,w)\bigr)=\lie{b}'_n\cdot(z,w)=\mbb{R}\delta(z,w)
\oplus\mbb{R}\xi_1(z,w)\oplus\dotsb\oplus\mbb{R}\xi_{n-1}(z,w)$. Elementary
considerations show that this real subspace of $\mbb{C}^n$ is totally real if
and only if the the matrix whose columns are given by the above vector fields
has non-zero determinant. Since one computes
\begin{equation*}
\det\begin{pmatrix}
2z & 2i w_1 & 2iw_2 & \cdots & 2iw_{n-1}\\
w_1 & 1 & 0 & \cdots & 0\\
w_2 & 0 & 1 &  & \vdots\\
\vdots & \vdots & & \ddots &0\\
w_{n-1} & 0 & \cdots & 0 & 1
\end{pmatrix}\\
=(-1)^{n-1}\Bigl(2z-2iw_1^2-\dotsb-2i w_{n-1}^2\Bigr),
\end{equation*}
the orbit of $B'_n\cdot(z,w)$ fails to be totally real if and only if
$z=i\sum_{k=1}^{n-1}w_k^2$ holds. Because of
\begin{equation*}
\im\left(i\sum_{k=1}^{n-1}w_k^2\right)-\norm{w}^2=-2\sum_{k=1}^{n-1}
\im(w_k)^2\leq0
\end{equation*}
such a point does not lie in $\mbb{H}_n$ which proves the claim.
\end{proof}

We have established the following fact.

\begin{corollary}\label{Cor:PerfectGroup}
Let $\xi\in\lie{b}_n$ be an arbitrary element. Then there exists an
$n$--dimensional subalgebra $\lie{b}'_n$ of $\lie{b}_n$ containing $\xi$ such
that the corresponding group $B'_n$ has only totally real orbits in $\mbb{H}_n$.
\end{corollary}

The same argument as in the proof of Proposition~\ref{Prop:TubeRealization}
applies to show the following

\begin{proposition}\label{Prop:TotallyReal}
Let $\xi\in\lie{b}_n$ be arbitrary. Then there exists a subgroup $B'_n\subset
B_n$ containing $\exp(\xi)$ such that $\mbb{H}_n$ is biholomorphically
equivalent to a $B'_n$--invariant domain in $(B'_n)^\mbb{C}$ where
$(B'_n)^\mbb{C}$ acts by left multiplication on itself.
\end{proposition}

\subsection{Quotients of the unit ball}

Let $\varphi$ be an automorphism of the unit ball $\mbb{B}_n$ which generates a
discrete subgroup $\Gamma\subset G$. The following proposition gives a
necessary condition for $X=\mbb{B}_n/\Gamma$ to be Stein.

\begin{proposition}\label{Prop:Embedding}
Let $\Omega$ be a domain in a Stein manifold $M$. Then $\Omega$ is Stein if and
only if some covering of $\Omega$ is Stein.
\end{proposition}

\begin{proof}
Let us assume that there is a covering of $\Omega$ which is Stein. It follows
from~\cite{S} that the universal covering $p\colon\widetilde{\Omega}\to\Omega$
is then Stein, too. If $\Omega$ is not Stein, then there exists a Hartogs figure
$(H,P)$ in $M$ such that $H\subset\Omega$ and $P\not\subset\Omega$ hold
(see~\cite{DocGr}). Since $H$ is simply connected, each component of $p^{-1}(H)$
is mapped biholomorphically onto $H$ by $p$. Let $\widetilde{H}$ be a component
of $p^{-1}(H)$ and write $s:=(p|_{\widetilde{H}})^{-1}\colon H\to\widetilde{H}$.
Since by assumption $\widetilde{\Omega}$ is a Stein manifold, we can embed it as
a closed submanifold into some $\mbb{C}^N$. Hence, the map $s$ extends to a map
$s\colon P\to\widetilde{\Omega}\subset\mbb{C}^N$. Thus the composition
$p\circ s\colon P\to\Omega\subset M$ is defined. Since $(p\circ s)|_H=\id_H$
holds, the continuation principle shows $p\circ s=\id_P$, which contradicts our
assumption $P\not\subset\Omega$.
\end{proof}

\begin{theorem}\label{Thm:BallQuotients}
Let $\varphi\in G$ be any automorphism generating a discrete subgroup $\Gamma$
of $G$. Then the quotient $X=\mbb{B}_n/\Gamma$ is a Stein space.
\end{theorem}

\begin{proof}
By virtue of Proposition~\ref{Prop:Reduction} we can assume that $\varphi$ lies
in the maximal split solvable subgroup $B_n\subset G$. Then we find an
$n$--dimensional closed subgroup $B_n'\subset B_n$ containing $\Gamma$ such that
each $B_n'$--orbit is totally real in $\mbb{B}_n$. By
Proposition~\ref{Prop:TotallyReal} we may embed $\mbb{B}_n$ as a
$B_n'$--invariant domain into $(B_n')^\mbb{C}$ where $(B_n')^\mbb{C}$ acts by
left multiplication on itself. Let $\mbb{C}_\Gamma$ be the complex one parameter
subgroup of $(B_n')^\mbb{C}$ which contains $\Gamma$. Since
$(B_n')^\mbb{C}/\Gamma$ is a $\mbb{C}^*$--principal bundle over
$(B_n')^\mbb{C}/\mbb{C}_\Gamma\cong\mbb{C}^{n-1}$, we conclude that
$(B_n')^\mbb{C}/\Gamma$ is a Stein manifold. Therefore the claim follows from
Proposition~\ref{Prop:Embedding}.
\end{proof}

\section{Existence of equivariant holomorphic submersions}

In~\cite{Pya} the $j$--invariant ideals of a normal $j$--algebra are
investigated. For the sake of completeness we indicate how the root space
decomposition of a normal $j$--algebra may be used to find a $j$--invariant
ideal which is isomorphic to the normal $j$--algebra of the unit ball.

\subsection{Existence of $j$--invariant ideals isomorphic to the unit ball}

Let $(\lie{s},j)$ be a normal $j$--algebra with gradation $\lie{s}=
\lie{s}_{-1}\oplus\lie{s}_{-1/2}\oplus\lie{s}_0$. We define $\lie{s}':=
\lie{s}'_{-1}\oplus\lie{s}'_{-1/2}\oplus\lie{s}'_0$ by
\begin{align*}
\lie{s}'_{-1}&:=\bigoplus_{k=1}^{r-1}\lie{s}_{\alpha_k}\oplus\bigoplus_{1\leq
k<l\leq r-1}\lie{s}_{(\alpha_l+\alpha_k)/2},\\
\lie{s}'_{-1/2}&:=\bigoplus_{k=1}^{r-1}\lie{s}_{\alpha_k/2},\\
\lie{s}'_0&:=\mbb{R}\eta_1\oplus\dotsb\oplus\mbb{R}\eta_{r-1}\oplus\bigoplus_{
1\leq k<l\leq r-1}\lie{s}_{(\alpha_l-\alpha_k)/2},
\end{align*}
i.\,e.\ $\lie{s}'$ is the direct sum of all root spaces in which the roots
$\alpha_r$, $\tfrac{1}{2}\alpha_r$ or $\tfrac{1}{2}(\alpha_r\pm\alpha_k)$
($1\leq k\leq r-1$) do not appear.

\begin{lemma}
The subspace $\lie{s}'$ is a $j$--invariant subalgebra of $\lie{s}$. Moreover,
there exists an $\omega'\in(\lie{s}')^*$ such that $(\lie{s}',j')$ is a
normal $j$--algebra where $j':=j|_{\lie{s}'}$.
\end{lemma}

\begin{proof}
The fact that $\lie{s}'$ is closed under the Lie bracket follows from the
properties of the root space decomposition and $j$--invariance is a direct
consequence of Proposition~\ref{Prop:Finestructure}~(3)-(5). Setting
$\omega':=\omega|_{\lie{s}'}$ the claim follows.
\end{proof}

Let $\pi\colon\lie{s}\to\lie{s}$ be the orthogonal projection onto $\lie{s}'$
with respect to $\langle\cdot,\cdot\rangle_\omega$.

\begin{lemma}
The map $\pi$ is a homomorphism of normal $j$--algebras algebras whose kernel is
given by
\begin{equation*}
\lie{b}:=\lie{s}_{\alpha_r}\oplus\bigoplus_{k=1}^{r-1}\lie{s}_{
(\alpha_r+\alpha_k)/2}\oplus\lie{s}_{\alpha_r/2}\oplus\mbb{R}\eta_r\oplus
\bigoplus_{k=1}^{r-1}\lie{s}_{(\alpha_r-\alpha_k)/2},
\end{equation*}
and hence induces an isomorphism $\lie{s}/\lie{b}\cong\lie{s}'$. In particular,
$\lie{b}$ is a $j$--invariant ideal in $\lie{s}$ and thus inherits the
structure of a normal $j$--algebra.
\end{lemma}

\begin{proof}
Using properties of the root space decomposition one checks directly that the
map $\pi$ preserves the Lie brackets. The kernel of $\pi$ is given by the
orthogonal complement of $\lie{s}'$ in $\lie{s}$ with respect to $\langle\cdot,
\cdot\rangle_\omega$ which in turn coincides with $\lie{b}$ by
Proposition~\ref{Prop:Finestructure}~(1). Since $\lie{s}'$ and $\lie{b}$ are
$j$--invariant, it follows that $\pi\circ j=j\circ\pi$ holds. This finishes
the proof.
\end{proof}

\begin{lemma}
The normal $j$--algebra $(\lie{b},j)$ is isomorphic to the normal
$j$--algebra of the unit ball.
\end{lemma}

\begin{proof}
One computes directly
\begin{equation*}
[\lie{b},\lie{b}]=\lie{s}_{\alpha_r}\oplus\bigoplus_{k=1}^{r-1}\lie{s}_{
(\alpha_r+\alpha_k)/2}\oplus\lie{s}_{\alpha_r/2}\oplus
\bigoplus_{k=1}^{r-1}\lie{s}_{(\alpha_r-\alpha_k)/2}.
\end{equation*}
Hence, $\mbb{R}\eta_r$ is maximal Abelian in $\lie{b}$ and in particular
$\lie{b}$ has rank one. The claim will follow if we show that
$[\lie{b},\lie{b}]$ is a Heisenberg algebra. For this one checks that
$\lie{s}_{\alpha_r}$ is the center of $[\lie{b},\lie{b}]$ and that the Lie
bracket
\begin{equation*}
[\cdot,\cdot]\colon\lie{b}_{-1/2}\times\lie{b}_{-1/2}\to\lie{s}_{\alpha_r}
\end{equation*}
defines a symplectic form on
\begin{equation*}
\lie{b}_{-1/2}:=\bigoplus_{k=1}^{r-1}
\lie{s}_{(\alpha_r+\alpha_k)/2}\oplus\lie{s}_{\alpha_r/2}\oplus\\
\bigoplus_{k=1}^{r-1} \lie{s}_{(\alpha_r-\alpha_k)/2}.\qedhere
\end{equation*}
\end{proof}

\begin{lemma}\label{Lem:Semidirect}
Let $\widehat{\pi}\colon S\to S'\cong S/B_m$ be the homomorphism on the group
level. The short exact sequence $1\to B_m\to S\to S'\to 1$ splits, i.\,e.\ $S$
is isomorphic to $S'\ltimes B_m$.
\end{lemma}

\begin{proof}
The claim follows from the fact that $\lie{s}'\hookrightarrow\lie{s}$ is a
homomorphism of Lie algebras and a section to $\pi$.
\end{proof}

\subsection{Geometric realization of the fibration}

In this subsection we view $\pi$ as a map $\lie{s}_{-1}^\mbb{C}\times
\lie{s}_{-1/2}\to(\lie{s}'_{-1})^\mbb{C}\times\lie{s}'_{-1/2}$ by restriction
and $\mbb{C}$--linear extension.

\begin{lemma}\label{Lem:ImageDomain}
The map $\pi$ maps $D_\lie{s}$ into $D_{\lie{s}'}$.
\end{lemma}

\begin{proof}
First we note that $\pi$ maps the base point $\xi_0=\xi_1+\dotsb+\xi_r$ onto the
base point $\xi_0'=\xi_1+\dotsb+\xi_{r-1}$. Since $\pi\colon\lie{s}\to\lie{s}'$
is a homomorphism of Lie algebras, it gives rise to a unique morphism
$\widehat{\pi}\colon S\to S'$ between the corresponding Lie groups such that
\begin{equation*}
\pi\bigl(\Ad(s)\xi\bigr)=\Ad\bigl(\widehat{\pi}(s)\bigr)\pi(\xi)
\end{equation*}
holds. Since $\pi$ also respects the grading of $\lie{s}$ and $\lie{s}'$, we
conclude that $\pi$ maps the cone $\Omega=\Ad(S_0)\xi_0$ onto the cone
$\Omega'=\Ad(S_0')\xi_0'$. Since the $\Omega$--Hermitian form $\Phi$ is
determined by the complex structure $J$ and the Lie bracket of $\lie{s}$ which
both are respected by $\pi$, we obtain
$\Phi'\bigl(\pi(\xi),\pi(\xi')\bigr)=\pi\Phi(\xi, \xi')$. This proves the claim.
\end{proof}

Choosing the base point $z_0:=(i\xi_0,0)\in D_\lie{s}$ we obtain the
diffeomorphism $S\to D_\lie{s}$, $s\mapsto s\cdot z_0$. Equipping $S$ with the
left invariant extension $J$ of $j$ this diffeomorphism becomes biholomorphic
(see Lemma~1.2 in~\cite{Is}). Let $B$ be the normal subgroup of $S$ with Lie
algebra $\lie{b}$ and let $S'$ be the analytic subgroup with Lie algebra
$\lie{s}'$. Note that $S'$ is isomorphic to $S/B$ via $\widehat{\pi}\colon S\to
S'$. The base point $z_0':=\pi(z_0)$ yields the isomorphism $S'\to
D_{\lie{s}'}$. Now we are in position to prove the main result of this section.

\begin{proposition}\label{Prop:GeomRealization}
The following diagram commutes:
\begin{equation*}
\xymatrix{
S\ar[r]^{\cong}\ar[d]_{\widehat{\pi}} & D_\lie{s}\ar[d]^{\pi}\\
S'\ar[r]_{\cong} & D_{\lie{s}'}.
}
\end{equation*}
It follows that $\pi\colon D_\lie{s}\to D_{\lie{s}'}$ is an $S$--equivariant
holomorphic submersion whose fibers are isomorphic to the unit ball.
\end{proposition}

\begin{proof}
We have to show that
\begin{equation*}
\pi(s\cdot z_0)=\widetilde{\pi}(s)\cdot z_0'
\end{equation*}
holds for all $s\in S$. In the proof of Lemma~\ref{Lem:ImageDomain} we have
already seen that this holds true for $s\in S_0$. Using the explicit
formula~\eqref{Eqn:Action} for the $S$--action on $D_\lie{s}$ one verifies the
claim for the whole group $S$.
\end{proof}

\begin{remark}
Let $\pi\colon D\to D'$ be the $S$--equivariant holomorphic submersion whose
fibers are biholomorphically equivalent to $\mbb{B}_m$. It follows from
Proposition~\ref{Prop:GeomRealization} that $\pi$ is a smooth principal bundle
with group $B_m$. If $\pi$ was a holomorphic fiber bundle, then by a result of
Royden (\cite{Roy}) it would be holomorphically trivial, i.\,e.\ $D\cong
D'\times\mbb{B}_m$. This shows that $\pi$ admits in general no local holomorphic
trivializations.
\end{remark}

\section{Proof of the main theorem}

\subsection{Equivariant fiber bundles}

In this subsection we present an auxiliary result concerning the quotient of an
equivariant fiber bundle. Since it seems to be hard to find an explicit
reference for it, we give a proof here. 

\begin{proposition}\label{Prop:EquivariantFiberBundle}
Let $p\colon B\to X$ be a fiber bundle with typical fiber $F$ and structure
group $S$. Let $G$ be a group acting on $B$ by bundle automorphisms. We assume
that the induced $G$--action on $X$ is free and proper. Then $G$ acts freely and
properly on $B$, and hence we obtain the commutative diagram
\begin{equation*}
\xymatrix{
B\ar[d]_p\ar[r] & B/G\ar[d]^{\overline{p}}\\
X\ar[r] & X/G.
}
\end{equation*}
The induced map $\overline{p}\colon B/G\to X/G$ is again a fiber bundle with the
same typical fiber and structure group.
\end{proposition}

\begin{proof}
We prove first that the map  $\overline{p}\colon B/\Gamma\to X/\Gamma$ admits
local trivializations. To see this let $U\subset X$ be an open set such that
there exists a trivialization $\varphi\colon U\times F\to p^{-1}(U)$. Shrinking
$U$ if necessary, we may assume that there exists a slice for the $G$--action
on $\widehat{U}:=G\cdot U$, i.\,e.\ that $\widehat{U}$ is $G$--equivariantly
isomorphic to $G\times S$ where $G$ acts on $G\times S$ by $g\cdot(g',x):=
(gg',x)$. It follows that $p^{-1}(S)$ is a slice for the $G$--action on
$p^{-1}(\widehat{U})$ (see~\cite{Pa2}), hence we obtain a $G$--equivariant
isomorphism $p^{-1}(\widehat{U})\cong_GG\times p^{-1}(S)$. Therefore the map
\begin{equation*}
\widehat{\varphi}\colon\widehat{U}\times F\to p^{-1}(\widehat{U}),\quad
\widehat{\varphi}(g\cdot x,y):=g\cdot\varphi(x,y),
\end{equation*}
with $g\in G$ and $x\in S$ is well-defined and hence a $G$--equivariant
trivialization. This implies that the map $\overline{p}\colon B/G\to X/G$ admits
local trivializations.

Since $G$ acts by bundle automorphisms on $B$, the transition functions between
different $G$--equivariant local trivializations $\widehat{\varphi}$ and
$\widehat{\psi}$ induce isomorphisms of $F$ given by the action of the
structure group $S$. Thus the structure group of the fiber bundle
$\overline{p}\colon B/G\to X/G$ is again given by $S$.
\end{proof}

\begin{example}
Let $G$ be a complex Lie group and let $H_1\subset H_2\subset G$ be closed
complex subgroups. According to Theorem~7.4 in~\cite{Ste} the natural map
$G/H_1\to G/H_2$ is a holomorphic fiber bundle with fiber $H_2/H_1$. The
structure group of this bundle is given by $H_2/(H_1)_0$ where $(H_1)_0$ denotes
the largest subgroup of $H_1$ which is normal in $H_2$. In particular, if $H_2$
is connected, then the structure group is connected. Moreover, the maps
$g'H_1\mapsto gg'H_1$, $g\in G$, are bundle automorphisms of $G/H_1\to G/H_2$.
Hence, we may apply Proposition~\ref{Prop:EquivariantFiberBundle} to any
subgroup $G'$ of $G$ which acts properly and freely on $G/H_2$ to obtain the
quotient bundle
\begin{equation*}
G'\backslash G/H_1\to G'\backslash G/H_2.
\end{equation*}
\end{example}

\subsection{Globalizing the submersion}

Let $D$ be a homogeneous Siegel domain and let $\pi\colon D\to D'$ be the
$S$--equivariant holomorphic submersion whose fibers are biholomorphically
equivalent to the unit ball $\mbb{B}_m$. By Lemma~\ref{Lem:TransitiveAction}
$S^\mbb{C}$ acts transitively on $S^\mbb{C}\cdot D=\mbb{C}^n$, hence we obtain
$S^\mbb{C}\cdot D=\mbb{C}^n\cong S^\mbb{C}/S^\mbb{C}_{z_0}$. Since $\mbb{C}^n$
is simply-connected, it follows that $S^\mbb{C}_{z_0}$ is connected. This
implies that $D^*:=S^\mbb{C}\cdot D$ is the universal globalization of the local
$S^\mbb{C}$--action on $D$. Similarly, $(D')^*:=(S')^\mbb{C}\cdot
D'=\mbb{C}^{n-m}$ is the universal globalization of the local
$(S')^\mbb{C}$--action on $D'$.

\begin{proposition}\label{Prop:BundleStructure}
There exists a unique $S^\mbb{C}$--equivariant map $\pi^*\colon D^*\to (D')^*$
which exhibits $D^*$ as a holomorphic fiber bundle over $(D')^*$ with typical
fiber $\mbb{H}_m^*=B_m^\mbb{C}\cdot\mbb{H}_m=\mbb{C}^m$. The structure group is
a connected complex Lie group.
\end{proposition}

\begin{proof}
Since the submersion $\pi\colon D\to D'$ is $S$--equivariant, we have $(D')^*=
(S')^\mbb{C}\cdot D'=S^\mbb{C}\cdot D'=\mbb{C}^{n-m}$. Since moreover $D^*\cong
S^\mbb{C}/S^\mbb{C}_{z_0}$ and $(D')^*\cong S^\mbb{C}/S^\mbb{C}_{z_0'}$, the
existence of $\pi^*$ follows from the fact that $S^\mbb{C}_{z_0}$ is contained
in $S^\mbb{C}_{z_0'}$. It is then immediate that $\pi^*\colon D^*\to (D')^*$ is
unique. Since $D^*$ and $(D')^*$ are simply-connected, the groups
$S^\mbb{C}_{z_0}$ and $S^\mbb{C}_{z_0'}$ are connected. Hence, it follows
from Theorem~7.4 in~\cite{Ste} that $D^*$ is a holomorphic fiber bundle over
$(D')^*$ with fiber $S^\mbb{C}_{z_0'}/S^\mbb{C}_{z_0}$ such that the structure
group is a connected complex Lie group. Since $B_m$ is a normal subgroup of $S$,
it lies in the $S$--isotropy of each point in $D'$. Hence, $B_m^\mbb{C}$ is a
normal subgroup of $S^\mbb{C}_{z_0'}$. Because of $S^\mbb{C}_{z_0}\cap
B_m^\mbb{C}=(B_m^\mbb{C})_{z_0}$ the inclusion $B_m^\mbb{C}\hookrightarrow
S^\mbb{C}_{z_0'}$ induces an isomorphism $B_m^\mbb{C}/(B_m^\mbb{C})_{z_0}\to
S^\mbb{C}_{z_0'}/S^\mbb{C}_{z_0}$ which proves that the fibers of $\pi^*$ are
isomorphic to $\mbb{H}_m^*=B_m^\mbb{C}\cdot\mbb{H}_m\cong
B_m^\mbb{C}/(B_m^\mbb{C})_{z_0}$.
\end{proof}

\begin{corollary}\label{Cor:RestrictedBundle}
We have $(\pi^*)^{-1}(D')=B_m^\mbb{C}\cdot D$. Hence, the restricted map
$\pi^*\colon B_m^\mbb{C}\cdot D\to D'$ is a holomorphic fiber bundle with
typical fiber $\mbb{H}_m^*$.
\end{corollary}

\begin{proof}
Let $\widetilde{z}\in(\pi^*)^{-1}(D')$ be given. By definition of $D^*$ there
exist a $g\in S^\mbb{C}$ and a $z\in D$ such that $\widetilde{z}=g\cdot z$
hold. Since $S'$ acts transitively on $D'$, we find a $g'\in S'$ such that
$g\cdot\pi(z)=\pi^*(\widetilde{z})=g'\cdot\pi(z)$ holds. Since $S'$ acts freely
on $D'$, we conclude $g(g')^{-1}\in B_m^\mbb{C}$. This shows that
$\widetilde{z}=g(g')^{-1}\cdot(g'\cdot z)\in B_m^\mbb{C}\cdot D$ holds. The
converse inclusion follows from the fact that $\pi^*$ is
$B_m^\mbb{C}$--invariant.
\end{proof}

\begin{corollary}\label{Prop:PrincBundle}
Let $B'_m$ be a subgroup of $B_m$ having only totally real orbits in
$\mbb{H}_m$. Then the universal globalization of the local
$(B'_m)^\mbb{C}$--action on $D$ is a $(B'_m)^\mbb{C}$--principal bundle over
$D'$.
\end{corollary}

\begin{proof}
Since $D'$ is a contractible Stein domain, we may apply Grauert's Oka principle
to the bundle $B_m^\mbb{C}\cdot D\to D'$ in order to obtain a
$B_m^\mbb{C}$--equivariant biholomorphism $B_m^\mbb{C}\cdot D\to D'\times
\mbb{H}_m^*$. Then Proposition~\ref{Prop:TotallyReal} implies that the
universal globalization of the local $(B'_m)^\mbb{C}$--action on $D$ is
isomorphic to $D'\times(B'_m)^\mbb{C}$ which proves the claim.
\end{proof}

\subsection{Proof of the main theorem}

Finally we are in position to prove that our main result.

\begin{theorem}
Let $D\subset\mbb{C}^n$ be a bounded homogeneous domain. Let $\varphi$ be an
automorphism of $D$ such that the group $\Gamma=\langle\varphi\rangle$ is
a discrete subgroup of $\Aut_\mathcal{O}(D)$. Then the quotient $X=D/\Gamma$ is
a Stein space.
\end{theorem}

\begin{proof}
Due to our reduction steps in Section~\ref{Sec:Reduction} we may assume that
$\varphi$ is contained in a maximal split solvable subgroup $S\subset G$.
Moreover, we assume that $D$ is realized as a Siegel domain in $\mbb{C}^n$ such
that $S$ acts by affine-linear transformations on $D$.

We prove this theorem by induction on $n=\dim_\mbb{C}D$. If $n=1$, then $D$ is
biholomorphically equivalent to the unit disc in $\mbb{C}$ and the claim
follows.

Let $n>1$ and let us assume that the claim is proven for every $n'<n$. Let
$\pi\colon D\to D'$ be the $S$--equivariant holomorphic submersion with fibers
isomorphic to $\mbb{H}_m$ onto the homogeneous Siegel domain
$D'\subset\mbb{C}^{n-m}$. In the first step we consider the case that $\Gamma$
is contained in the normal subgroup $B_m$, hence that $\pi$ is
$\Gamma$--invariant. By Corollary~\ref{Cor:PerfectGroup} there exists an
$m$--dimensional
subgroup $B_m'$ of $B_m$ which contains $\Gamma$ such that every $B_m'$--orbit
in $\mbb{H}_m$ is totally real. Applying Corollary~\ref{Prop:PrincBundle} we
see that the universal globalization $D^*$ of the local $(B'_m)^\mbb{C}$--action
on $D$ is a $(B_m')^\mbb{C}$--principal bundle over $D'$ which must be
holomorphically trivial. Therefore $D^*/\Gamma$ is a Stein manifold and
Proposition~\ref{Prop:Embedding} applies to show that $X$ is Stein.

If $\Gamma$ is not contained in $B_m$, then we obtain a proper $\Gamma$--action
on $D'$. Since $\Gamma$ normalizes the group $B_m^\mbb{C}$, it follows that
$\Gamma$ acts by bundle automorphisms on the holomorphic fiber bundle
$B_m^\mbb{C}\cdot D\to D'$. We conclude from
Proposition~\ref{Prop:EquivariantFiberBundle} that the quotient bundle
\begin{equation*}
\left(B_m^\mbb{C}\cdot D\right)/\Gamma\to D'/\Gamma
\end{equation*}
is a holomorphic fiber bundle with fiber $\mbb{H}_m^*$ and connected structure
group $S^\mbb{C}_{z_0'}/(S^\mbb{C}_{z_0})_0$. Since the base is Stein by our
induction hypothesis, a result of Matsushima and Morimoto (Theorem~6
in~\cite{MatMo}) implies that $\left(B_m^\mbb{C}\cdot D\right)/\Gamma$ is Stein,
hence that $X$ is Stein. This finishes the proof.
\end{proof}

\end{document}